\documentclass{article}

 \PassOptionsToPackage{numbers, compress}{natbib}

 \usepackage[preprint]{nips_2018}

\usepackage[utf8]{inputenc} 
\usepackage[T1]{fontenc}    
\usepackage{hyperref}       
\usepackage{url}            
\usepackage{booktabs}       
\usepackage{amsfonts}       
\usepackage{nicefrac}       
\usepackage{microtype}      

\usepackage[utf8]{inputenc} 
\usepackage[english]{babel} 

\usepackage{graphicx}
\usepackage{mathtools, cuted}
\usepackage{hyperref}
\usepackage{amsmath}
\usepackage{amssymb}
\usepackage{mathtools}
\usepackage{amsthm}

\newtheorem{definition}{Definition}[section] 
\newtheorem{theorem}{Theorem}[section]
\newtheorem{corollary}{Corollary}[section]
\newtheorem{lemma}[theorem]{Lemma} 
\newtheorem{proposition}[theorem]{Proposition} 
\newtheorem{remark}[theorem]{Remark}

\title{Reciprocal Sequences as CM Sequences}

\author{
  Reza Rezaie and X. Rong Li \\
 Department of Electrical Engineering\\
 University of New Orleans\\
New Orleans, LA 70148 \\
  \texttt{rrezaie@uno.edu} and \texttt{xli@uno.edu} \\
}

\begin{document}

\maketitle

\begin{abstract}
The conditionally Markov (CM) sequence contains several classes, including the reciprocal sequence. Reciprocal sequences have been widely used in many areas of engineering, including image processing, acausal systems, intelligent systems, and intent inference. In this paper, the reciprocal sequence is studied from the CM sequence point of view, which is different from the viewpoint of the literature and leads to more insight into the reciprocal sequence. Based on this viewpoint, new results, properties, and easily applicable tools are obtained for the reciprocal sequence. The nonsingular Gaussian (NG) reciprocal sequence is modeled and characterized from the CM viewpoint. It is shown that a NG sequence is reciprocal if and only if it is both $CM_L$ and $CM_F$ (two special classes of CM sequences). New dynamic models are presented for the NG reciprocal sequence. These models (unlike the existing one, which is driven by colored noise) are driven by white noise and are easily applicable. As a special reciprocal sequence, the Markov sequence is also discussed. Finally, it can be seen how all CM sequences, including Markov and reciprocal, are unified.
\end{abstract}


\section{Introduction}

The CM process is a very large set and it contains reciprocal and Markov processes as two important special cases. Reciprocal processes have been used in many different areas of science and engineering (e.g., \cite{Levy_1}--\cite{Krener1}), where stochastic processes more general than Markov processes are needed. Applications of reciprocal processes in image processing were discussed in \cite{Picci}--\cite{Picci2}. Based on quantized state space, \cite{Fanas1}--\cite{Fanas2} used finite-state reciprocal sequences for detection of anomalous trajectory pattern and intent inference. The idea of the reciprocal process was utilized in \cite{Simon} for intent inference in intelligent interactive displays of vehicles. \cite{DD_Conf}--\cite{DW_Conf} proposed some classes of CM sequences (including reciprocal) for trajectory modeling. In \cite{Krener1}, the behavior of acausal systems was described using reciprocal processes.

This paper looks at the reciprocal sequence from the viewpoint of the CM sequence, which is a larger set of sequences. This point of view, which is different from that of the literature on the reciprocal sequence, reveals more properties of the reciprocal sequence and leads to a better insight and easily applicable results.

The notion of CM processes was introduced in \cite{Mehr} for Gaussian processes based on mean and covariance functions. Also, some Gaussian CM processes were defined based on conditioning at the first index (time) of the CM interval. \cite{Mehr} considered Gaussian processes being nonsingular on the interior of the time interval. Also, stationary Gaussian CM processes were characterized, and construction of some non-stationary Gaussian CM processes was discussed. \cite{ABRAHAM} extended the definition of Gaussian CM processes (presented in \cite{Mehr}) to the general (Gaussian/non-Gaussian) case. In addition, it was shown how continuous time Gaussian CM processes can be represented by a Wiener process and an uncorrelated Gaussian random vector \cite{ABRAHAM}, \cite{Mehr}. In \cite{CM_Part_I}, different (Gaussian/non-Gaussian) CM sequences based on conditioning at the first or the last time of the CM interval were defined, and (stationary/non-stationary) NG CM sequences were studied. Also, dynamic models and characterizations of NG CM sequences were presented in \cite{CM_Part_I}. Based on a valuable observation, \cite{ABRAHAM} commented on the relationship between the Gaussian CM process and the Gaussian reciprocal process. Following the comment of \cite{ABRAHAM}, \cite{Carm}--\cite{Carm2} obtained some results about continuous time Gaussian reciprocal processes.

The notion of reciprocal processes was introduced in \cite{Berstein} connected to a problem posed by E. Schrodinger \cite{Schrodinger_1}--\cite{Schrodinger_2}. Later, reciprocal processes were studied more in \cite{Slepian}--\cite{CM_Part_II_A} and others. A dynamic model and a characterization of the NG reciprocal sequence were presented in \cite{Levy_Dynamic} (which is the most significant paper on the Gaussian reciprocal sequence related to our work). It was shown that the evolution of the NG reciprocal sequence can be described by a second-order nearest-neighbor model driven by locally correlated dynamic noise \cite{Levy_Dynamic}. That model can be considered a generalization of the Markov model. However, due to the correlation of the dynamic noise as well as the nearest-neighbor structure, it is not necessarily easy to apply \cite{DD_Conf}--\cite{DW_Conf}. Based on this second-order model, a characterization of the NG reciprocal sequence in terms of its covariance matrix was obtained in \cite{Levy_Dynamic}. 

Consider stochastic sequences defined over the interval $[0,N]=$  $\lbrace 0,1,\ldots,N \rbrace$. For convenience, let the index be time. A sequence is Markov if and only if (iff) conditioned on the state at any time $k$, the subsequences before and after $k$ are independent. A sequence is reciprocal iff conditioned on the states at any two times $k_1$ and $k_2$, the subsequences inside and outside the interval $[k_1,k_2]$ are independent. In other words, inside and outside are independent given the boundaries. A sequence is CM over $[k_1,k_2]$ iff conditioned on the state at any time $k_1$ ($k_2$), the sequence is Markov over $(k_1,k_2]$ ($[k_1,k_2)$). The Markov sequence and the reciprocal sequence are two important classes of the CM sequence.

The contributions of this paper are as follows. The reciprocal sequence is viewed as a special CM sequence. Studying, modeling, and characterizing the reciprocal sequence from this viewpoint are different from those of \cite{Levy_Dynamic} and in the literature. This fruitful angle has several advantages. It provides more insight into the reciprocal sequence and its relation to other CM sequences. New results regarding the reciprocal sequence, as a special CM sequence, are obtained. More specifically, the relationship between the (Gaussian/non-Gaussian) reciprocal sequence and the CM sequence is presented. It is shown that a NG sequence is reciprocal iff it is both $CM_L$ and $CM_F$. In other words, it is discussed how different classes of CM sequences contribute to the construction of the reciprocal sequence. New dynamic models for the NG reciprocal sequence are obtained based on CM models. These models driven by white (rather than colored) noise are easily applicable. Also, it is discussed under what conditions these models govern NG Markov sequences. 

The paper is organized as follows. Section \ref{Section_Definitions_Preliminaries} reviews definitions of CM, reciprocal, and Markov sequences as well as some results required for later sections. Results obtained in Section \ref{Section_Reciprocal_Characterziation} and Section \ref{Section_Reciprocal_Dynamic} are for NG sequences except Theorem \ref{CM_iff_Reciprocal}, which is for the general (Gaussian/non-Gaussian) case. In Section \ref{Section_Reciprocal_Characterziation}, the reciprocal sequence is studied from the CM viewpoint. In Section \ref{Section_Reciprocal_Dynamic}, based on CM models, new dynamic models for the NG reciprocal sequence are obtained. Section \ref{Section_Summary_Conclusions} contains a summary and conclusions.

\section{Definitions and Preliminaries}\label{Section_Definitions_Preliminaries}

\subsection{Conventions}\label{Subsection_Convention}

Throughout the paper we consider stochastic sequences defined over the interval $[0,N]$, which is a general discrete index interval. For convenience this discrete index is called time. The following conventions are used throughout the paper.  
\begin{align*}
[i,j]& \triangleq \lbrace i,i+1,\ldots ,j-1,j \rbrace
\end{align*}
\begin{align*}
(i,j)& \triangleq \lbrace i+1,i+2,\ldots ,j-2,j-1 \rbrace\\
[x_k]_{i}^{j} & \triangleq \lbrace x_k, k \in [i,j] \rbrace\\
[x_k] & \triangleq [x_k]_{0}^{N}\\
i,j,k_1,k_2, l_1, l_2& \in [0,N], k_1<k_2, i<j
\end{align*}
where $k$ in $[x_k]_i^j$ is a dummy variable. The symbol ``$ \setminus $" is used for set subtraction. $C_{l_1,l_2}$ is a covariance function. $C$ is the covariance matrix of the whole sequence $[x_k]$. Also, $0$ may denote a zero scalar, vector, or matrix, as is clear from the context. $F(\cdot | \cdot)$ denotes the conditional cumulative distribution function (CDF). For a matrix $A$, $A_{[r_1:r_2,c_1:c_2]}$ denotes its submatrix consisting of (block) rows $r_1$ to $r_2$ and (block) columns $c_1$ to $c_2$ of $A$.

The abbreviations ZMNG and NG are used for ``zero-mean nonsingular Gaussian" and ``nonsingular Gaussian", respectively.

\subsection{Definitions and Notations}\label{Definitions}

Formal definitions of CM sequences can be found in \cite{CM_Part_I}. Here we present the definitions in a simple language. A sequence $[x_k]$ is $[k_1,k_2]$-$CM_c, c \in \lbrace k_1,k_2 \rbrace$ (i.e., CM over $[k_1,k_2]$) iff conditioned on the state at any time $k_1$ ($k_2$), the sequence is Markov over $(k_1,k_2]$ ($[k_1,k_2)$). The above definition is equivalent to the following lemma \cite{CM_Part_I}.

\begin{lemma}\label{CMc_CDF}
$[x_k]$ is $[k_1,k_2]$-$CM_c, c \in \lbrace k_1,k_2 \rbrace$, iff 
 \begin{align}
 F(\xi _k|[x_{i}]_{k_1}^{j},x_{c})=F(\xi _k|x_j,x_c)\label{CDF_1}
 \end{align} 
for every $j,k \in [k_1,k_2], j<k$, or equivalently,
 \begin{align}
 F(\xi _k|[x_{i}]_{j}^{k_2},x_c)=F(\xi _k|x_j,x_{c})\label{CDF_2}
 \end{align} 
for every $k,j \in [k_1,k_2], k< j$, and every $\xi _k \in \mathbb{R}^d$, where $d$ is the dimension of $x_k$.  

\end{lemma}

The interval $[k_1,k_2]$ of the $[k_1,k_2]$-$CM_c$ sequence is called the \textit{CM interval} of the sequence.

\begin{remark}\label{R_CMN}
For the forward direction, we have
\begin{align*}
[k_1,k_2]\text{-}CM_c=\left\{ \begin{array}{cc} 
[k_1,k_2]\text{-}CM_F & \text{if} \quad  c=k_1\\ \relax
[k_1,k_2]\text{-}CM_L & \text{if} \quad  c=k_2
\end{array} \right.
\end{align*}
where the subscript ``$F$" or ``$L$" is used because the conditioning is at the \textit{first} or \textit{last} time of the CM interval. 

\end{remark}

\begin{remark}\label{R_CMN_2}
The CM interval of a sequence is dropped if it is the whole time interval: our shorthand for the $[0,N]$-$CM_c$ sequence is $CM_c$.

\end{remark}

In the forward direction, a $CM_0$ ($CM_N$) sequence is called a $CM_F$ ($CM_L$) sequence. 

Corresponding to different values of $k_1, k_2 \in [0,N]$, and $c \in \lbrace k_1,k_2 \rbrace$, there are different classes of CM sequences. For example, $CM_F$ and $(0,N]$-$CM_L$ are two classes. By a $CM_F \cap (0,N]$-$CM_L$ sequence we mean a sequence which is both $CM_F$ and $(0,N]$-$CM_L$.

The reciprocal sequence is as follows. A sequence is reciprocal iff the subsequences inside and outside the interval $[k_1,k_2]$ are independent conditioned on the boundaries $x_{k_1}$ and $x_{k_2}$ ($\forall k_1,k_2 \in [0,N]$). The above definition is equivalent to the following lemma.

\begin{lemma}\label{CDF}
$[x_k]$ is reciprocal iff 
\begin{align}
F(\xi _k|[x_{i}]_{0}^{j},[x_i]_l^N)=F(\xi _k|x_j,x_l)\label{CDF_1}
\end{align} 
for every $j,k,l \in [0,N]$ ($j < k < l$), and every $\xi _k \in \mathbb{R}^d$, where $d$ is the dimension of $x_k$.  

\end{lemma}

\begin{lemma}\label{CDF_M}
$[x_k]$ is Markov iff 
 \begin{align}
 F(\xi _k|[x_{i}]_{0}^{j})=F(\xi _k|x_j)\label{CDF_1_M}
 \end{align} 
for every $j,k \in [0,N], j<k$, or equivalently,
 \begin{align}
 F(\xi _k|[x_{i}]_{j}^{N})=F(\xi _k|x_j)\label{CDF_2_M}
 \end{align} 
for every $k,j \in [0,N], k< j$, and every $\xi _k \in \mathbb{R}^d$, where $d$ is the dimension of $x_k$.  

\end{lemma}

\subsection{Preliminaries}\label{Preliminaries}

We review some results required later \cite{CM_Part_I}. 

\begin{definition}\label{CMc_Matrix}
A symmetric positive definite matrix is called $CM_L$ if it has form $\eqref{CML}$ and $CM_F$ if it has form $\eqref{CMF}$.
\begin{align}\left[
\begin{array}{ccccccc}
A_0 & B_0 & 0 & \cdots & 0 & 0 & D_0\\
B_0' & A_1 & B_1 & 0 & \cdots & 0 & D_1\\
0 & B_1' & A_2 & B_2 & \cdots & 0 & D_2\\
\vdots & \vdots & \vdots & \vdots & \vdots & \vdots & \vdots\\
0 & \cdots & 0 & B_{N-3}' & A_{N-2}  & B_{N-2} & D_{N-2}\\
0 & \cdots & 0 & 0 & B_{N-2}' & A_{N-1} & B_{N-1}\\
D_0' & D_1' & D_2' & \cdots & D_{N-2}' & B_{N-1}' & A_N
\end{array}\right]\label{CML}
\end{align}
\begin{align}\left[
\begin{array}{ccccccc}
A_0 & B_0 & D_2 & \cdots & D_{N-2} & D_{N-1} & D_{N}\\
B_0' & A_1 & B_1 & 0 & \cdots & 0 & 0\\
D_2' & B_1' & A_2 & B_2 & \cdots & 0 & 0\\
\vdots & \vdots & \vdots & \vdots & \vdots & \vdots & \vdots\\
D_{N-2}' & \cdots & 0 & B_{N-3}' & A_{N-2}  & B_{N-2} & 0\\
D_{N-1}' & \cdots & 0 & 0 & B_{N-2}' & A_{N-1} & B_{N-1}\\
D_{N}' & 0 & 0 & \cdots & 0 & B_{N-1}' & A_N
\end{array}\right]\label{CMF}
\end{align}

\end{definition}
$A_k$, $B_k$, and $D_k$ are matrices in general.  

To refer to both $CM_L$ and $CM_F$ matrices we call them $CM_c$. A $CM_c$ matrix for $c=N$ is $CM_L$ and for $c=0$ is $CM_F$.

\begin{theorem}\label{CML_Characterization} 
A NG sequence with covariance matrix $C$ is $CM_c$ iff $C^{-1}$ has the $CM_c$ form.

\end{theorem}

\begin{proposition}\label{CML_k1k2_Characterization_1}
A NG $[x_k]$ with $A$ being the inverse of its covariance matrix is

(i) $[0,k_2]$-$CM_c$ ($k_2 \in [1,N-1]$) iff $\Delta_{A_{22}}$ has the $CM_c$ form, where
\begin{align}
\Delta _{A_{22}}&=A_{11}-A_{12}A_{22}^{-1}A_{12}'\label{DA_22}
\end{align}
$A_{11}=A_{[1:k_2+1,1:k_2+1]}$, $A_{22}=A_{[k_2+2:N+1,k_2+2:N+1]}$, and $A_{12}=A_{[1:k_2+1,k_2+2:N+1]}$.

(ii) $[k_1,N]$-$CM_c$ ($k_1 \in [1,N-1]$) iff $\Delta_{A_{11}}$ has the $CM_c$ form, where
\begin{align}
\Delta _{A_{11}}&=A_{22}-A_{12}'A_{11}^{-1}A_{12}\label{DA_11}
\end{align}
$A_{11}=A_{[1:k_1,1:k_1]}$, $A_{22}=A_{[k_1+1:N+1,k_1+1:N+1]}$, and $A_{12}=$ $A_{[1:k_1,k_1+1:N+1]}$. 

\end{proposition}

A positive definite matrix $A$ is called a $[0,k_2]$-$CM_c$ ($[k_1,N]$-$CM_c$) matrix if $\Delta_{A_{22}}$ ($\Delta_{A_{11}}$) in $\eqref{DA_22}$ ($\eqref{DA_11}$) has the $CM_c$ form.

Theorem \ref{CML_Dynamic_Forward_Proposition} and Proposition \ref{CMF_Dynamic_Backward_Proposition} present forward and backward $CM_c$ dynamic models.

\begin{theorem}\label{CML_Dynamic_Forward_Proposition}
A ZMNG $[x_k]$ with covariance function $C_{l_1,l_2}$ is $CM_c$ iff it is governed by
\begin{align}
x_k=G_{k,k-1}x_{k-1}+G_{k,c}x_c+e_k, \quad k \in (0,N] \setminus \lbrace c \rbrace
\label{CML_Dynamic_Forward}
\end{align}
where $[e_k]$ is a zero-mean white Gaussian sequence with nonsingular covariances $G_k$, along with boundary condition\footnote{$e_0$ and $e_N$ in $\eqref{CML_Forward_BC1}$ are not necessarily the same as $e_0$ and $e_N$ in $\eqref{CML_Forward_BC2}$. Just for simplicity we use the same notation.}
\begin{align}
&x_0=e_0, \quad x_c=G_{c,0}x_0+e_c \, \, (\text{for} \,\, c=N)\label{CML_Forward_BC1}
\end{align}
or equivalently
\begin{align}
&x_c=e_c, \quad x_0=G_{0,c}x_c+e_0 \, \, (\text{for} \, \, c=N) \label{CML_Forward_BC2}
\end{align}

\end{theorem}

\begin{proposition}\label{CMF_Dynamic_Backward_Proposition}
A ZMNG $[x_k]$ with covariance function $C_{l_1,l_2}$ is $CM_c$ iff its evolution is governed by 
\begin{align}
x_k=G^B_{k,k+1}x_{k+1}+G^B_{k,c}x_c+e^{B}_k, \quad k \in [0,N) \setminus \lbrace c\rbrace
\label{CML_Dynamic_Backward}
\end{align}
where $[e^{B}_k]$ is a zero-mean white Gaussian sequence with nonsingular covariances $G^B_k$, along with the boundary condition
\begin{align}
x_N=e^{B}_N, \quad x_{c}=G^{B}_{c,N}x_N+e^{B}_{c} \, \, (\text{for} \, \, c=0) \label{CML_Backward_BC1}
\end{align}
or equivalently
\begin{align}
x_c=e^{B}_c, \quad x_{N}=G^{B}_{N,c}x_c+e^{B}_{N} \, \, (\text{for} \, \, c=0) \label{CML_Backward_BC2}
\end{align}

\end{proposition}

\section{Reciprocal Sequence Characterization}\label{Section_Reciprocal_Characterziation}

The relationship between the CM sequence and the reciprocal sequence is presented in Theorem \ref{CM_iff_Reciprocal} for the general (Gaussian/non-Gaussian) case. Proofs can be found in \cite{CM_Part_II_A}.

\begin{theorem}\label{CM_iff_Reciprocal}
$[x_k]$ is reciprocal iff it is

(i) $[k_1,N]$-$CM_F$, $\forall k_1 \in [0,N]$, and $CM_L$

or equivalently

(ii) $[0,k_2]$-$CM_L$, $\forall k_2 \in [0,N]$, and $CM_F$

\end{theorem}

In \cite{ABRAHAM} a part of the conditions of Theorem \ref{CM_iff_Reciprocal} (i.e., $[k_1,N]$-$CM_F$, $\forall k_1 \in [0,N]$) was mentioned, but the other part (i.e., $CM_L$) was overlooked. The condition presented in \cite{ABRAHAM} is not sufficient for a Gaussian process to be reciprocal \cite{CM_Part_II_A}.

From the proof of Theorem \ref{CM_iff_Reciprocal} (see \cite{CM_Part_II_A}), a sequence that is $[k_1,N]$-$CM_F$ ($\forall k_1 \in [0,N]$) and $CM_L$ or  equivalently is $[0,k_2]$-$CM_L$ ($\forall k_2 \in [0,N]$) and $CM_F$ is actually $[k_1,N]$-$CM_F$ and $[0,k_2]$-$CM_L$ ($\forall k_1, k_2 \in [0,N]$). It means that a (Gaussian/non-Gaussian) sequence is reciprocal iff it is $[k_1,N]$-$CM_F$ and $[0,k_2]$-$CM_L$ ($\forall k_1, k_2 \in [0,N]$). This was pointed out for the Gaussian case in \cite{Carm}. However, \cite{Carm} did not discuss if the condition presented in \cite{ABRAHAM} is sufficient even for the Gaussian case.

A characterization of the NG reciprocal sequence was presented in \cite{Levy_Dynamic}. Below that characterization is obtained from the CM viewpoint, which is simple and different from the proof presented in \cite{Levy_Dynamic}.

\begin{theorem}\label{Reciprocal_Characterization} 
A NG sequence with covariance matrix $C$ is reciprocal iff $C^{-1}$ is cyclic tridiagonal as $\eqref{Cyclic_Tridiagonal}$.
\begin{align}\left[
\begin{array}{ccccccc}
A_0 & B_0 & 0 & \cdots & 0 & 0 & D_0\\
B_0' & A_1 & B_1 & 0 & \cdots & 0 & 0\\
0 & B_1' & A_2 & B_2 & \cdots & 0 & 0\\
\vdots & \vdots & \vdots & \vdots & \vdots & \vdots & \vdots\\
0 & \cdots & 0 & B_{N-3}' & A_{N-2}  & B_{N-2} & 0\\
0 & \cdots & 0 & 0 & B_{N-2}' & A_{N-1} & B_{N-1}\\
D_0' & 0 & 0 & \cdots & 0 & B_{N-1}' & A_N
\end{array}\right]\label{Cyclic_Tridiagonal}
\end{align}

\end{theorem}
\begin{proof}
Necessity: By Theorem \ref{CM_iff_Reciprocal}, characterization of the NG reciprocal sequence is the same as that of the NG sequence being $[0,k_2]$-$CM_L$, $\forall k_2 \in [0,N]$ and $CM_F$. Let $[x_k]$ be a NG sequence (with the covariance matrix $C$), which is $[0,k_2]$-$CM_L$, $ \forall k_2 \in [0,N]$ and $CM_F$. By Theorem \ref{CML_Characterization}, $C^{-1}$ is cyclic (block) tri-diagonal, because a matrix being both $CM_L$ and $CM_F$ is cyclic tri-diagonal. 

Sufficiency: Assume that the inverse of the covariance matrix ($C^{-1}$) of a NG $[x_k]$ is cyclic (block) tri-diagonal. A cyclic tri-diagonal matrix has the $CM_F$ and the $[0,k_2]$-$CM_L$ forms $\forall k_2 \in [0,N]$. Therefore, by Theorem \ref{CML_Characterization} and Proposition \ref{CML_k1k2_Characterization_1}, $[x_k]$ is $CM_F$ and $[0,k_2]$-$CM_L$, $\forall k_2 \in [0,N]$. Thus, by Theorem \ref{CM_iff_Reciprocal}, $[x_k]$ is reciprocal.     
\end{proof}

The following corollary of Theorem \ref{Reciprocal_Characterization} has an important implication about the relationship between the NG CM sequence and the NG reciprocal sequence. In addition, Corollary \ref{CMLCMF_Reciprocal} demonstrates the significance of studying the reciprocal sequence from the CM viewpoint.

\begin{corollary}\label{CMLCMF_Reciprocal}
A NG sequence is reciprocal iff it is both $CM_L$ and $CM_F$.
\end{corollary}

By Corollary \ref{CMLCMF_Reciprocal}, a NG sequence being both $CM_L$ and $CM_F$ is $[k_1,k_2]$-$CM_L$ and $[k_1,k_2]$-$CM_F$, $\forall k_1,k_2 \in [0,N]$.

Note that a characterization of the NG Markov sequence is as follows \cite{Ackner}.
\begin{remark}\label{Markov_Characterization}
A NG sequence with covariance matrix $C$ is Markov iff $C^{-1}$ is tri-diagonal as $\eqref{Tridiagonal}$.
\begin{align}\left[
\begin{array}{ccccccc}
A_0 & B_0 & 0 & \cdots & 0 & 0 & 0\\
B_0' & A_1 & B_1 & 0 & \cdots & 0 & 0\\
0 & B_1' & A_2 & B_2 & \cdots & 0 & 0\\
\vdots & \vdots & \vdots & \vdots & \vdots & \vdots & \vdots\\
0 & \cdots & 0 & B_{N-3}' & A_{N-2}  & B_{N-2} & 0\\
0 & \cdots & 0 & 0 & B_{N-2}' & A_{N-1} & B_{N-1}\\
0 & 0 & 0 & \cdots & 0 & B_{N-1}' & A_N
\end{array}\right]\label{Tridiagonal}
\end{align}

\end{remark}

By Theorem \ref{CML_Characterization} and Proposition \ref{CML_k1k2_Characterization_1} we have

\begin{itemize}

\item Special case: A NG sequence with covariance matrix $C$ is $CM_L \cap [N-3,N]$-$CM_F$ iff $C^{-1}$ is given by $\eqref{CML}$ with $D_{N-2}=0$.

\item Special case: A NG sequence with covariance matrix $C$ is $CM_L \cap [N-4,N]$-$CM_F$ iff $C^{-1}$ is given by $\eqref{CML}$ with $D_{N-3}=D_{N-2}=0$.

\item General case: A NG sequence with covariance matrix $C$ is $CM_L \cap [k_1,N]$-$CM_F$ iff $C^{-1}$ is given by $\eqref{CML}$ with
\begin{align*}
D_{k_1+1}=D_{k_1+2}=\cdots =D_{N-3}=D_{N-2}=0
\end{align*}

\item Important special case: A NG sequence with covariance matrix $C$ is $CM_L \cap CM_F$ iff $C^{-1}$ is given by $\eqref{CML}$ with 
\begin{align*}
D_1=D_2=\cdots =D_{N-2}=0
\end{align*}
which is actually the reciprocal characterization (Corollary \ref{CMLCMF_Reciprocal}).

\end{itemize}

It is thus seen how the characterizations (i.e., $C^{-1}$) gradually change from $CM_L$ to reciprocal.

\section{Reciprocal $CM_c$ Dynamic Models}\label{Section_Reciprocal_Dynamic}

The reciprocal sequence is a special $CM_c$ sequence. Thus, it can be modeled by a $CM_c$ model. Also, one can find conditions for a $CM_c$ model to govern a reciprocal sequence. In the following, some models are presented for the NG reciprocal sequence based on $CM_c$ models.

\begin{theorem}\label{CML_R_Dynamic_Forward_Proposition}
A ZMNG $[x_k]$ with covariance function $C_{l_1,l_2}$ is reciprocal iff it is governed by $\eqref{CML_Dynamic_Forward}$ along with $\eqref{CML_Forward_BC1}$ or $\eqref{CML_Forward_BC2}$, and 
\begin{align}
G_k^{-1}G_{k,c}=G_{k+1,k}'G_{k+1}^{-1}G_{k+1,c}
\label{CML_Condition_Reciprocal}
\end{align}
$\forall k \in (0,N-1)$ for $c=N$, or $\forall k \in (1,N)$ for $c=0$. Moreover, $[x_k]$ is Markov iff additionally we have, for $c=N$,
\begin{align}
G_N^{-1}G_{N,0}=G_{1,N}'G_{1}^{-1}G_{1,0}\label{CML_M_BC1}
\end{align}
for $\eqref{CML_Forward_BC1}$ or equivalently
\begin{align}
G_0^{-1}G_{0,N}=G_{1,0}'G_1^{-1}G_{1,N}\label{CML_M_BC2}
\end{align}
for $\eqref{CML_Forward_BC2}$; for $c=0$, we have 
\begin{align}
G_{N,0}=0\label{CMF_M}
\end{align}

\end{theorem}

 The Markov sequence is a subset of the reciprocal sequence, and the reciprocal sequence is a subset of the $CM_c$ sequence. A $CM_c$ model is a complete description (i.e., necessary and sufficient) for the $CM_c$ sequence (Theorem \ref{CML_Dynamic_Forward_Proposition}). Theorem \ref{CML_R_Dynamic_Forward_Proposition} shows under what conditions a $CM_c$ model is a complete description of the reciprocal sequence, and under what conditions a $CM_c$ model is a complete description of the Markov sequence. Thus, Theorem \ref{CML_R_Dynamic_Forward_Proposition} unifies $CM_c$, reciprocal, and Markov sequences.

Theorem \ref{CML_R_Dynamic_Forward_Proposition} can be presented in another way.

\begin{corollary}\label{CML_R_Dynamic_Forward_Corollary}
Model $\eqref{CML_Dynamic_Forward}$ along with $\eqref{CML_Forward_BC1}$ or $\eqref{CML_Forward_BC2}$ governs a ZMNG reciprocal sequence iff the matrix
\begin{align}
A=\mathcal{G}'G^{-1}\mathcal{G}\label{A_Forward}
\end{align}
is cyclic tridiagonal, where $G=\text{diag}(G_0,G_1,\ldots, G_N)$, for $c=N$ the matrix $\mathcal{G}$ is, for $\eqref{CML_Forward_BC1}$,
\begin{align}\label{L_1}
\left[ \begin{array}{cccccc}
I & 0 & 0 &  \cdots & 0 & 0\\
-G_{1,0} & I & 0 &  \cdots & 0 & -G_{1,N}\\
0 & -G_{2,1} & I & 0 & \cdots & -G_{2,N}\\
\vdots & \vdots & \vdots & \vdots & \vdots & \vdots \\
0 & 0 & \cdots & -G_{N-1,N-2} & I & -G_{N-1,N}\\
-G_{N,0} & 0 & 0 &  \cdots & 0 & I
\end{array}\right]
\end{align}
$\mathcal{G}$ is, for $\eqref{CML_Forward_BC2}$, 
\begin{align}\label{L_2}
\left[ \begin{array}{cccccc}
I & 0 & 0 &  \cdots & 0 & -G_{0,N}\\
-G_{1,0} & I & 0 &  \cdots & 0 & -G_{1,N}\\
0 & -G_{2,1} & I & 0 & \cdots & -G_{2,N}\\
\vdots & \vdots & \vdots & \vdots & \vdots & \vdots \\
0 & 0 & \cdots & -G_{N-1,N-2} & I & -G_{N-1,N}\\
0 & 0 & 0 &  \cdots & 0 & I
\end{array}\right]
\end{align}
and for $c=0$, $\mathcal{G}$ is
\begin{align}\label{F}
\left[ \begin{array}{cccccc}
I & 0 & 0 &  \cdots & 0 & 0\\
-2G_{1,0} & I & 0 &  \cdots & 0 & 0\\
-G_{2,0} & -G_{2,1} & I & 0 & \cdots & 0\\
\vdots & \vdots & \vdots & \vdots & \vdots & \vdots \\
-G_{N-1,0} & 0 & \cdots & -G_{N-1,N-2} & I & 0\\
-G_{N,0} & 0 & 0 &  \cdots & -G_{N,N-1} & I
\end{array}\right]
\end{align}
In addition, the sequence is Markov iff $A$ is tri-diagonal.

\end{corollary}

A $CM_c$ model governing a reciprocal (Markov) sequence is called a reciprocal (Markov) $CM_c$ model.

The following proposition presents backward models for the reciprocal sequence.

\begin{proposition}\label{CMF_R_Dynamic_Backward_Proposition}
A ZMNG sequence $[x_k]$ with covariance function $C_{l_1,l_2}$ is reciprocal iff it is governed by $\eqref{CML_Dynamic_Backward}$ along with $\eqref{CML_Backward_BC1}$ or $\eqref{CML_Backward_BC2}$ and
\begin{align}
(G^B_{k+1})^{-1}G^B_{k+1,c}=(G^B_{k,k+1})'(G^B_{k})^{-1}G^B_{k,c}
\label{CMF_Condition_Reciprocal_B}
\end{align}
$\forall k \in (0,N-1)$ for $c=0$, or $\forall k \in [0,N-2)$ for $c=N$. Moreover, $[x_k]$ is Markov iff additionally we have, for $c=0$,
\begin{align}
(G^B_0)^{-1}G^B_{0,N}=(G^B_{N-1,0})'(G^B_{N-1})^{-1}G^B_{N-1,N}\label{CML_B_M_BC1}
\end{align}
for $\eqref{CML_Backward_BC1}$, or equivalently
\begin{align}
(G^B_N)^{-1}G^B_{N,0}=(G^B_{N-1,N})'(G^B_{N-1})^{-1}G^B_{N-1,0}
\label{CML_B_M_BC2}
\end{align}
for $\eqref{CML_Backward_BC2}$; for $c=N$, we have
\begin{align}
G^B_{0,N}=0\label{CML_B_M}
\end{align}

\end{proposition}

Proposition \ref{CMF_R_Dynamic_Backward_Proposition} can be presented as follows.

\begin{corollary}\label{CMF_R_Dynamic_Backward_Corollary}
Backward model $\eqref{CML_Dynamic_Backward}$ along with $\eqref{CML_Backward_BC1}$ or $\eqref{CML_Backward_BC2}$ govern a ZMNG reciprocal sequence iff the matrix $A$ is cyclic tri-diagonal with
\begin{align}
A=(\mathcal{G}^B)'(G^B)^{-1}\mathcal{G}^B\label{A_Backward}
\end{align}
where $G^B=\text{diag}(G^B_0, G^B_1,\ldots, G^B_N)$ and for $c=0$ the matrix $\mathcal{G}^B$ is
\begin{align}\label{L_B_1}
\left[ \begin{array}{cccccc}
I & 0 & 0 &  \cdots & 0 & -G^{B}_{0,N}\\
-G^B_{1,0} & I & -G^B_{1,2} &  \cdots & 0 & 0\\
-G^B_{2,0} & 0 & I & -G^B_{2,3} & \cdots & 0\\
\vdots & \vdots & \vdots & \vdots & \vdots & \vdots \\
-G^B_{N-1,0} & 0 & \cdots & 0 & I & -G^{B}_{N-1,N}\\
0 & 0 & 0 &  \cdots & 0 & I
\end{array}\right]
\end{align}
for $\eqref{CML_Backward_BC1}$, and $\mathcal{G}^B$ is
\begin{align}\label{L_B_2}
\left[ \begin{array}{cccccc}
I & 0 & 0 &  \cdots & 0 & 0\\
-G^B_{1,0} & I & -G^B_{1,2} &  \cdots & 0 & 0\\
-G^B_{2,0} & 0 & I & -G^B_{2,3} & \cdots & 0\\
\vdots & \vdots & \vdots & \vdots & \vdots & \vdots \\
-G^B_{N-1,0} & 0 & \cdots & 0 & I & -G^{B}_{N-1,N}\\
-G^{B}_{N,0} & 0 & 0 &  \cdots & 0 & I
\end{array}\right]
\end{align}
for $\eqref{CML_Backward_BC2}$, and for $c=N$, $\mathcal{G}^B$ is
\begin{align}\label{F_B}
\left[ \begin{array}{cccccc}
I & -G^B_{0,1} & 0 &  \cdots & 0 & -G^B_{0,N}\\
0 & I & -G^B_{1,2} &  \cdots & 0 & -G^B_{1,N}\\
0 & 0 & I & -G^B_{2,3} & \cdots & -G^B_{2,N}\\
\vdots & \vdots & \vdots & \vdots & \vdots & \vdots \\
0 & 0 & \cdots & 0 & I & -2G^{B}_{N-1,N}\\
0 & 0 & 0 &  \cdots & 0 & I
\end{array}\right]
\end{align}
In addition, the sequence is Markov iff $A$ in $\eqref{A_Backward}$ is tri-diagonal.

\end{corollary}

In order to derive a recursive estimator for the NG reciprocal sequence, \cite{Moura} manipulated the reciprocal model of \cite{Levy_Dynamic} to obtain simple recursive models. It can be seen that the resultant models in \cite{Moura} are actually in the form of the $CM_c$ models presented above. It, once more, demonstrates the significance of studying the reciprocal sequence from the CM viewpoint.

\section{Summary and Conclusions}\label{Section_Summary_Conclusions}

As an important special conditionally Markov (CM) sequence, the reciprocal sequence has been studied from the CM viewpoint. This fruitful angle, which is different from that of the literature, has advantages. It reveals several properties of the reciprocal sequence. Also, it leads to easily applicable results. As a result, Markov, reciprocal, and other CM sequences are unified. 

The relationship between the CM sequence and the reciprocal sequence for the general (Gaussian/non-Gaussian) case has been presented. It was shown that a NG sequence is reciprocal iff it is both $CM_L$ and $CM_F$. It demonstrates the key role of the $CM_L$ and $CM_F$ sequences in the study of the reciprocal sequence from the CM viewpoint. 

Based on $CM_c$ models, several models have been presented for the evolution of the NG reciprocal sequence. Unlike the existing reciprocal model of \cite{Levy_Dynamic}, these models are driven by white noise, which leads to simplicity and easy applicability. 

The significance and benefits of studying the reciprocal sequence as a special CM sequence have been demonstrated by the obtained results.

CM dynamic models (including the reciprocal $CM_L$ model) were studied further in \cite{CM_Part_II_B_Conf}--\cite{CM_Explicitly}, where some approaches for their parameter design in application were presented.

\subsubsection*{Acknowledgments}

Research was supported by NASA Phase03-06 through grant NNX13AD29A.

\end{document}